\documentclass[reqno]{amsart}
\usepackage{amssymb}
\usepackage{hyperref}
\usepackage[english]{babel}

\title[Spectra originated  Fredholm theory and Browder's theorem ]
{Spectra originated from Fredholm theory and Browder's theorem}

\author[  M. AMOUCH, M. KARMOUNI, A. TAJMOUATI   ]
{ M. AMOUCH, M. KARMOUNI, A. TAJMOUATI }

\address{ Abdelaziz TAJMOUATI and Mohammed KARMOUNI  \newline
 Sidi Mohamed Ben Abdellah
 Univeristy
 Faculty of Sciences Dhar Al Mahraz Fez, Morocco.}
\email{abdelaziz.tajmouati@usmba.ac.ma}
\email{mohammed.karmouni@usmba.ac.ma}
\address{Mohamed AMOUCH \newline
Department of Mathematics
University Chouaib Doukkali,
Faculty of Sciences, Eljadida.
24000, Eljadida, Morocco.}
\email{mohamed.amouch@gmail.com}

\subjclass[2010]{47A53, 47B10}
\keywords{Generalized Kato decomposition, pseudo B-Fredholm operator, Browder's theorem, Riesz operator,  Commuting perturbation.}

\newtheorem{theorem}{Theorem}[section]
\newtheorem{definition}{Definition}[section]
\newtheorem{remark}{Remark}
\newtheorem{lemma}{Lemma}[section]
\newtheorem{proposition}{Proposition}[section]
\newtheorem{corollary}{Corollary}[section]
\newtheorem{example}{Example}

\begin{document}

\maketitle

\begin{abstract}
We give a new characterization of Browder's theorem  through equality between the pseudo B-Weyl  spectrum and the  generalized Drazin spectrum. Also, we will give conditions under which pseudo B-Fredholm and pseudo B-Weyl spectrum introduced in \cite{BO} and \cite{ZZ} become stable under commuting Riesz perturbations.

\end{abstract}

\section{Introduction and Preliminaries}

Throughout, $X$ denotes a complex Banach space, $\mathcal{B}(X)$
the Banach algebra of all bounded linear operators on $X$,
let $I$ be the identity operator, and for $T\in \mathcal{B}(X)$ we
denote by $T^*$,  $R(T)$,  $ R^{\infty}(T)=\bigcap_{n\geq0}R(T^n)$, $\rho(T)$,
$\sigma(T)$, $\sigma_{p}(T)$,  $\sigma_{ap}(T)$ and  $\sigma_{su}(T)$
respectively the adjoint, the range, the hyper-range,  the resolvent set, the spectrum, the point spectrum, the approximate point spectrum and the
surjectivety spectrum of $T$.
%%%%%%%%%%%%%%%%%%%%%%%%%%%%%%%%%%%%%%%%%%%%%%%%%%%%%%%%%%%%%%%%%%%%%%%%%%%%%%%%%%%%%%%%%%%%%%%%%%%%%%%%%%%%%%%

An operator $T\in\mathcal{B}(X)$ is said to be semi-regular,  if $R(T)$ is closed
and $N(T)\subseteq R^{\infty}(T)$. For subspaces $M$, $N$ of $X$ we write $M\subseteq^{e}N$ ($M$ is essentially contained in $N$) if there exists a finite-dimensional subspace $F\subset X$ such that $M\subseteq N+F$. $T\in\mathcal{B}(X)$ is said to be essentially semi-regular,  if $R(T)$ is closed
and $N(T)\subseteq^{e} R^{\infty}(T)$. The corresponding spectra are
the semi-regular spectrum $\sigma_{se}(T)$  and the essentially semi-regular spectrum $\sigma_{es}(T)$ defined by
$$\sigma_{se}(T)=\{\lambda\in\mathbb{C}: T-\lambda I\mbox{  is not semi-regular }\}$$
$$\sigma_{es}(T)=\{\lambda\in\mathbb{C}: T-\lambda I\mbox{  is not essentially semi-regular } \},  \mbox{ see  \cite{Aie}}$$
%%%%%%%%%%%%%%%%%%%%%%%%%%%%%%%%%%%%%%%%%%%%%%%%%%%%%%%%%%%%%%%%%%%%%%%%%%%%%%%%%%%%%%%%%%%%%%%%%%%%%%%%%%%%%%%%
%%%%%%%%%%%%%%%%%%%%%%%%%%%%%%%%%%

Let $E$ be a subset of $X$. $E$ is said $T$-invariant if $T(E)\subseteq E$.  We say that $T$ is completely reduced by the pair $(E, F)$ and we denote $(E, F)\in Red(T)$ if $E$ and $F$ are two closed $T$-invariant subspaces of $X$ such that
$X=E\oplus F$. In this case we write $T=T_{\shortmid {E}}\oplus T_{\shortmid {F}}$ and we say that $T$ is the direct sum of $T_{\shortmid {E}}$ and $T_{\shortmid {F}}$.
%%%%%%%%%%%%%%%%%%%%%%%%%%%%%%%%%%%%%%%%%%%%%%%%%%%%%%%%%%%%%%%%%%%%%%%%%%%%%%%%%%%%%%%%%%%%%%%%%%%%%%%%%%%%%%%
In the other hand, recall that an operator $T \in \mathcal{B}(X)$ admits a generalized Kato decomposition, ( GKD for short ),
if there exists $(X_1, X_2)\in Red(T)$ such that  $T_{\shortmid {X_{1}}}$ is semi-regular and $T_{\shortmid {X_{2}}}$
 is quasi-nilpotent, in this case $T$ is said a pseudo Fredholm operator. If we assume in the definition above that $T_{\shortmid {X_{2}}}$ is nilpotent, then $T$ is said to be of Kato type. Clearly, every semi-regular operator is of Kato type
 and a quasi-nilpotent operator has a GKD, see \cite{lab, Mb1} for more information
 about generalized Kato decomposition.\\
%%%%%%%%%%%%%%%%%%%%%%%%%%%%%%%%%%%%%%%%%%%%%%%%%%%%%%%%%%%%%%%%%%%%%%%%%%%%%%%%%%%%%%%%%%%%%%%%%%%%%%%%%%%
\indent Recall that $T \in \mathcal{B}(X)$ is said to be quasi-Fredholm if there exists $d\in \mathbb{N}$ such that
\begin{enumerate}
	\item $R(T^{n})\cap N(T)=R(T^{d})\cap N(T)$ \ \ for all $n\geq d;$
	\item $R(T^{d})\cap N(T)$ and $R(T)+ N(T^{d})$ are closed in $X.$
\end{enumerate}

An operator is quasi-Fredholm if it is quasi-Fredholm of some degree $d$.
Note that semi-regular operators are quasi-Fredholm of degree $0$
and by results of Labrousse \cite{lab}, in the case of Hilbert spaces, the set of quasi-Fredholm
operators coincides with the set of Kato type operators.
\noindent For every bounded operator $T \in \mathcal{B}(X)$, let us define
the essential quasi-Fredholm spectrum and generalized Kato spectrum respectively by:
$$\sigma_{eq}(T):=\{\lambda \in \mathbb{C} : T- \lambda I \text{ is not quasi-Fredholm} \};$$
$$\sigma_{gK}(T):=\{\lambda \in \mathbb{C} : T- \lambda I \text{ does not admit a generalized Kato decomposition} \}.$$
It is know that $\sigma_{gK}(T)$ is always a compact subsets of the complex plane contained in the spectrum $\sigma(T)$ of $T$ \cite[Corollary 2.3]{JZ}.
Note that $\sigma_{gK}(T)$ is not necessarily non-empty.
For example, all quasi-nilpotent operator has an empty generalized Kato spectrum, see \cite{JZ, QH} for more information about $\sigma_{gK}(T).$
%%%%%%%%%%%%%%%%%%%%%%%%%%%%%%%%%%%%%%%%%%%%%%%%%%%%%%%%%%%%%%%%%%%%%%%%%%%%%%%%%%%%%%%%%%%%%%%%%%%%%%%%%%%

A bounded linear operator is called an upper semi-Fredholm
(resp, lower semi Fredholm) if $dim N(T)<\infty \mbox{ and  } R(T) \mbox{  is closed }$
(resp, $codim R(T) <\infty$). $T$ is semi-Fredholm if it is a lower or upper semi-Fredholm operator.
The index of a semi-Fredholm operator $T$ is defined by $ind(T):= dim N(T)- codim R(T).$ Also,
$T$ is a Fredholm operator if it is a lower and upper semi-Fredholm operator,
and $T$ is called a Weyl operator if it is a Fredholm of index zero.\\
The essential  and Weyl spectra  of $T$  are closed and  defined  by :
$$\sigma_{e}(T)=\{\lambda\in \mathbb{C}:  T-\lambda I \mbox{  is not a  Fredholm operator}\};$$
$$\sigma_{W}(T)=\{\lambda\in \mathbb{C}:  T-\lambda I  \mbox{ is not  a Weyl operator}\}.$$

Recall  that an operator $R\in\mathcal{B}(X)$ is said to be Riesz if $R-\mu I$
is Fredholm for every non-zero complex number $\mu$.
Of course compact and quasi-nilpotent operators are particular cases of Riesz operators.

Let $T\in\mathcal{B}(X)$, the ascent of $T$ is defined by $a(T)=min\{p\in\mathbb{N}: N(T^p)=N(T^{p+1})\}$, if such $p$ does not exist we let $a(T)=\infty$. Analogously the descent of $T$ is $d(T)=min\{q\in\mathbb{N}: R(T^q)=R(T^{q+1})\}$, if such $q$ does not exist we let $d(T)=\infty$ \cite{LT}. It is well known that if both $a(T)$ and $d(T)$ are finite then $a(T)=d(T)$ and we have the decomposition $X=R(T^p)\oplus N(T^p)$ where $p=a(T)=d(T)$.\\
An operator   $T\in \mathcal{B}(X)$ is upper semi-Browder if $T$ is upper semi-Fredholm and $a(T)<\infty$. If  $T\in \mathcal{B}(X)$ is lower semi-Fredholm and $d(T)<\infty$ then $T$ is lower semi-Browder. $T$ is called Browder operator if it is a lower and upper Browder operator.

An operator $T\in \mathcal{B}(X)$  is said to be B-Fredholm,
if for some integer $n\geq0$ the range $R(T^n)$ is closed and $T_n$,
the restriction of $T$ to $R(T^n)$ is a Fredholm operator.
This class of operators, introduced and studied by Berkani et al. in a series of papers
extends the class of semi-Fredholm operators.
$T$ is said to be a B-Weyl operator if
$T_n$ is a Fredholm operator of index zero. The B-Fredholm and   B-Weyl spectra are defined by
$$\sigma_{BF}(T)=\{\lambda\in \mathbb{C}:\>\> T-\lambda I \mbox{ is not  B-Fredholm} \};$$
$$\sigma_{BW}(T)=\{\lambda\in \mathbb{C}:\>\> T-\lambda I \mbox{ is not  B-Weyl}\}.$$
Note that $T$ is a B-Fredholm operator if  there exists $(X_1, X_2)\in  Red(T)$  such that  $T_{\shortmid {X_{1}}}$ is Fredholm
and $T_{\shortmid {X_{2}}}$ is nilpotent, see \cite[Theorem 2.7]{B1}.
Also, $T$ is a B-Weyl operator if and only if
$T_{\shortmid X_{1}}$  is a Weyl operator and $T_{\shortmid X_{2}}$  is a nilpotent operator.\\
\indent More recently, B-Fredholm and B-Weyl operators were generalized to pseudo B-Fredholm and pseudo B-Weyl, see
\cite{BO, ZZ}, precisely,
$T$ is a pseudo B-Fredholm operator, if  there exists  $(X_1, X_2)\in Red(T)$
such that  $T_{\shortmid X_{1}}$ is a Fredholm  operator and $T_{\shortmid X_{2}}$ is a quasi-nilpotent operator.
$T$ is said to be  pseudo B-Weyl operator if there exists  $(X_1, X_2)\in Red(T)$  such that  $T_{\shortmid X_{1}}$ is a Weyl  operator and $T_{\shortmid X_{2}}$ is a quasi-nilpotent operator.
The  pseudo B-Fredholm and  pseudo B-Weyl spectra are  defined by:
$$\sigma_{pBF}(T)=\{\lambda\in \mathbb{C}:\>\> T-\lambda I \mbox{ is  not  pseudo B-Fredholm} \};$$
$$\sigma_{gD\mathcal{W}}(T)=\{\lambda\in \mathbb{C}:\>\> T-\lambda I \mbox{ is  not   pseudo B-Weyl} \}.$$
Let $T\in \mathcal{B}(X)$,
$T$ is said to be Drazin invertible if there exist a positive integer $k$ and an operator  $S\in\mathcal{B}(X)$ such that  $$ST=TS, \,\,\,T^{k+1}S=T^k\,\, \,\,and\,\,  S^2T=S.$$ Which is also equivalent to  the fact that $T=T_{1}\oplus T_{2}$; where $T_{1}$ is invertible  and $T_{2}$ is nilpotent. The Drazin spectrum is defined by $$\sigma_{D}(T)=\{\lambda\in \mathbb{C}: T-\lambda I  \mbox{ is not Drazin invertible } \}.$$ The concept of Drazin invertible operators has been generalized by Koliha \cite{K}. In fact, $T\in \mathcal{B}(X)$ is generalized Drazin invertible if and only if $0\notin acc(\sigma(T))$, where   $acc(\sigma(T))$ is the set of   accumulation points  of $\sigma(T)$. This  is also equivalent to the fact that
there exists  $(X_1, X_2)\in Red(T)$  such that  $T_{\shortmid {X_{1}}}$ is invertible and $T_{\shortmid {X_{2}}}$ is quasi-nilpotent.
The  generalized Drazin spectrum is defined by $$\sigma_{gD}(T)=\{\lambda\in \mathbb{C}: T-\lambda I \mbox{ is not  generalized Drazin invertible } \}.$$
%%%%%%%%%%%%%%%%%%%%%%%%%%%%%%%%%%%%%%%%%%%%%%%%%%%%%%%%%%%%%%%%%%%%%%%%%%%%%%%%%%%%%%%%%%%%%%%%%%%%%%%%%%%%%%%%%%%%%%%%%%%%%%%%%%%%%%%%%%%%%%%%%%%%%
The concept of analytical core for an operator has been introduced  by $Vrbov\grave{a}$  in \cite{vrb} and study by Mbekhta \cite{Mb1, Mb22}, that is the following set:
\begin{center}
 $K(T) = \{x \in X \>: \> \exists (x_n)_{n\geq0} \subset  X \>\mbox{ and }\> \delta > 0 \>\mbox{ such that } \>x_{0}=x,\> T x_n = x_{n-1} \>\forall n \geq1 \>\mbox{ and }\> \|x_{n}\| \leq \delta^{n}\|x\|\}$
 \end{center}
 The quasi-nilpotent part of $T$, $H_{0}(T)$ is given by :
  $$H_{0}(T):=\{x\in X; r_{T}(x)=0\} \mbox{ where  } r_{T}(x)=\displaystyle{\lim_{n\rightarrow+\infty}} ||T^nx||^{\frac{1}{n}}.$$
\indent In \cite{CZ}, M D. Cvetkovi\'{c} and  S\v{C}. \v{Z}ivkovi\'{c}-Zlatanovi\'{c} introduced  and studied  a new
concept of  generalized Drazin invertibility of bounded operators as a generalization of generalized Drazin invertible operators. In fact,  an operator $T\in\mathcal{B}(X)$ is said to be  generalized Drazin bounded below if $H_0(T)$ is closed and complemented with a subspace $M$ in $X$ such that $(M, H_0(T))\in Red(T)$ and  $ T(M)$ is closed which is equivalent to  there exists $(M, N)\in Red(T)$ such that   $T_{\shortmid {M}}$ is  bounded below and $T_{\shortmid N}$ is  quasi-nilpotent, see \cite[Theorem 3.6]{CZ}.
An operator $T\in\mathcal{B}(X)$ is said to be  generalized Drazin surjective if $K(T)$ is closed and complemented with a subspace $N$ in $X$ such that $ N\subseteq H_0(T)$  and $(K(T), N)\in Red(T)$ which is equivalent to there exists $(M, N)\in Red(T)$ such that   $T_{\shortmid {M}}$ is surjective and $T_{\shortmid N}$ is  quasi-nilpotent, see \cite[Theorem 3.7]{CZ}.\\
The   generalized Drazin bounded below and surjective spectra of $T\in \mathcal{B}(X)$ are defined respectively by:
    $$\sigma_{gD\mathcal{M}}(T)=\{\lambda\in\mathbb{C},\,\, T-\lambda  I \mbox{  is  not   generalized   Drazin bounded below} \};$$
    $$\sigma_{gD\mathcal{Q}}(T)=\{\lambda\in\mathbb{C},\,\, T-\lambda I \mbox{ is  not   generalized  Drazin surjective} \}.$$
     From \cite{CZ}, we have:
    $$\sigma_{gD}(T)=\sigma_{gD\mathcal{M}}(T)\cup \sigma_{gD\mathcal{Q}}(T).$$

 %%%%%%%%%%%%%%%%%%%%%%%%%%%%%%%%%%%%%%%%%%%%%%%%%%%%%%%%%%%%%%%%%%%%%%%%%%%%%%%%%%%%%%%%%%%%%%%%%%%%%%%%%%%%%%%%%%

 As a continuation of works \cite{AmB, AmZ, AZ1, BO, CZ, ZZ}, we will study  various spectra originated from Fredholm theory and related to Drazin spectrum.
After given preliminaries results, in the second section of this work,
we characterize  the equality between the pseudo B-Weyl  spectrum and generalized Drazin spectrum   by means of the Browder's theorem.
Also,  we will give serval necessary and sufficient conditions for $T$ to have equality between the spectra  originated from Fredholm theory and Drazin invertibility.
In the same direction as our work \cite{TAK},  we will give  conditions under which pseudo B-Fredholm and pseudo B-Weyl spectrum
 are stable under commuting Riesz perturbations.
 In section four, we will prove that we can perturb a pseudo B-Fredholm (resp. pseudo Fredholm) operator $T\in\mathcal{B}(X)$
by a bounded operator $S$ commuting with $T$ to obtain a
Fredholm (resp. semi-regular operator ) $T+S.$

\section{ On pseudo semi B-Fredholm (Weyl) operators }
In the following, we introduce the definition of pseudo upper B-Fredholm, pseudo lower B-Fredholm, generalized Drazin lower semi-Weyl,  generalized Drazin upper  semi-Weyl and pseudo semi B-Fredholm operators.

\begin{definition}\cite{CZ}
An operator $T\in \mathcal{B}(X)$ is said to be  pseudo upper B-Fredholm if  there exist  two $T$-invariant  closed subspaces $X_{1}$ and $X_{2}$ of $X$  such that $X=X_1\oplus X_2$ and $T_{\shortmid X_{1}}$ is upper semi-Fredholm operator and $T_{\shortmid X_{2}}$ is quasi-nilpotent. If $ind(T_{\shortmid X_{1}})\leq 0$,  $T$ is said to be generalized Drazin upper semi-Weyl.
\end{definition}

\begin{definition}\cite{CZ}
An operator $T\in \mathcal{B}(X)$ is said to be  pseudo lower  B-Fredholm if  there exist  two $T$-invariant  closed subspaces $X_{1}$ and $X_{2}$ of $X$  such that $X=X_1\oplus X_2$ and $T_{\shortmid X_{1}}$ is lower semi-Fredholm operator and $T_{\shortmid X_{2}}$ is quasi-nilpotent. If $ind(T_{\shortmid X_{1}})\leq 0$,  $T$ is said to be generalized Drazin lower semi-Weyl.
\end{definition}
\begin{definition}
 We say that $T\in\mathcal{B}(X)$ is pseudo semi B-Fredholm if  $T$ is  pseudo lower  B-Fredholm  or  pseudo upper  B-Fredholm.
 \end{definition}
It is clear that  $T$ is a pseudo B-Fredholm  operator if and only if $T$  is a pseudo lower semi B-Fredholm operator   and  pseudo upper semi B-Fredholm operator. In the same way  $T$ is  pseudo B-Weyl  if and only if $T$  is generalized Drazin lower semi-Weyl and generalized Drazin upper  semi-Weyl.
The   generalized Drazin lower semi-Weyl and generalized Drazin upper semi-Weyl  spectra   of $T\in \mathcal{B}(X)$ are defined respectively by:
    $$\sigma_{gD\mathcal{W-}}(T)=\{\lambda\in\mathbb{C},\,\, T-\lambda  I \mbox{  is  not   generalized   Drazin lower semi-Weyl} \};$$
    $$\sigma_{gD\mathcal{W+}}(T)=\{\lambda\in\mathbb{C},\,\, T-\lambda I \mbox{ is  not   generalized  Drazin upper semi-Weyl} \}.$$
     From \cite{CZ}, we have:
    $$\sigma_{gD\mathcal{W}}(T)=\sigma_{gD\mathcal{W+}}(T)\cup \sigma_{gD\mathcal{W-}}(T);$$
The pseudo upper and lower B-Fredholm spectra  of $T\in \mathcal{B}(X)$ are defined respectively by:
    $$\sigma_{puBF}(T)=\{\lambda\in\mathbb{C},\,\, T-\lambda  I \mbox{  is  not pseudo upper B-Fredholm } \};$$
    $$\sigma_{plBF}(T)=\{\lambda\in\mathbb{C},\,\, T-\lambda I \mbox{ is  not pseudo  lower B-Fredholm} \}.$$
     Also, from \cite{CZ}, we have:
    $$\sigma_{pBF}(T)=\sigma_{puBF}(T)\cup \sigma_{plBF}(T).$$

 The following results gives some relationship between  pseudo upper/lower B-Fredholm operator in terms of generalized Drazin invertibility.

\begin{proposition}\label{aaaa}
Let $T\in \mathcal{B}(X)$. If there exists $(N, F)\in Red (T)$ such that $codim F<\infty$, $dim N<\infty$ and $T_{\shortmid F}$ is  generalized Drazin bounded below, then $T$ is pseudo upper B-Fredholm.
\end{proposition}

\begin{proof}
If  there exists $(N, F)\in Red (T)$ such that  $codim F<\infty$, $dim N<\infty$ and $T_{\shortmid F}$ is  generalized Drazin bounded below,
then $X=F\oplus N$. Since $T_{\shortmid F}$ is generalized Drazin bounded below, then there exist two closed $T$-invariant subspaces $F_{1}$ and $F_{2}$ of $F$  such that $F=F_1\oplus F_2$, $T_{\shortmid F_{1}}$ is bounded below  and $T_{\shortmid F_{2}}$ is quasi-nilpotent, then $X=F_1\oplus F_2\oplus N.$
Let $M=F_1\oplus N$, $T(M)=T(F_1) +T(N)$, since $T_{\shortmid F_{1}}$ is bounded below, then $T(F_{1})$ is closed.
Since $dim N<\infty,$ then $T(M)$ is closed. Now we have $$N(T_{ \shortmid M})=N(T_{\shortmid F_{1}})\oplus N(T_{ \shortmid N})=N(T_{ \shortmid N})\subseteq N,$$
because $T_{\shortmid F_{1}}$ is bounded below. Therefore, $T_{ \shortmid M}$ is upper Fredholm and $T_{\shortmid F_{2}}$ is quasi-nilpotent.
Thus $T$ is  pseudo upper B-Fredholm.
\end{proof}

%%%%%%%%%%%%%%%%%%%%%%%%%%%%%%%%%%%%%%%%%%%%%%

\begin{proposition}\label{aabb}
Let $T\in \mathcal{B}(X)$.  If $T$ is pseudo lower B-Fredholm,  then there exists $F\subseteq X$ such that $codim F<\infty$ and $T_{\shortmid F}$ is  generalized Drazin surjective.

Conversely, If there exists $(N, F)\in Red (T)$ such that $codim F<\infty$, $dim N<\infty$  and $T_{\shortmid F}$ is  generalized Drazin surjective, then
 $T$ is pseudo lower B-Fredholm
\end{proposition}

\begin{proof}
If $T$ is pseudo lower B-Fredholm, then there exist  two closed $T$-invariant subspaces $X_{1}$ and $X_{2}$ of $X$
such that $X=X_1\oplus X_2$ and $T_1=T_{\shortmid X_{1}}$ is lower semi-Fredholm and $T_2=T_{\shortmid X_{2}}$ is quasi-nilpotent.
Since $T_1$ is lower semi-Fredholm, then $codim R(T_1)<\infty,$ hence there exists $N\subseteq X_1$  such that, $dim N<\infty $ and $X_1=R(T_1)\oplus N.$
Thus, $X=N\oplus R(T_1)\oplus X_2.$
Let $F=R(T_1)\oplus X_2$, then $codim F<\infty$ and  $T_{\shortmid R(T_1)}$ is surjective and $T_2$ is quasi-nilpotent, so $T$ is  generalized Drazin surjective.

Conversely, if there exists  $(N, F)\in Red (T)$ such that $codim F<\infty$, $dim N<\infty$  and $T_{\shortmid F}$ is  generalized Drazin surjective. Since $T_{\shortmid F}$ is  generalized Drazin surjective, then there exists two closed $T$-invariant subspaces $F_{1}$ and $F_{2}$ of $F$  such that
$F=F_1\oplus F_2$ and $T_{\shortmid F_{1}}$ is surjective  and $T_{\shortmid F_{2}}$ is quasi-nilpotent, then $X=F_1\oplus F_2\oplus N.$
Let $M=F_1\oplus N$,  since $T_{\shortmid F_{1}}$ is surjective, then $T_{\shortmid F_{1}}$ is  lower Fredholm. Since $T_{\shortmid N}$ is finite rank operator, so $T_{\shortmid M}=T_{\shortmid F_{1}}\oplus T_{\shortmid N}$ is lower Fredholm. Therefore, $T_{\shortmid F_{1}}\oplus T_{\shortmid N}$ is lower Fredholm and $T_{\shortmid F_{2}}$ is quasi-nilpotent. So, $T$ is  pseudo lower B-Fredholm.
\end{proof}
%%%%%%%%%%%%%%%%%%%%%%%%%%%%%%%%%%%%%%%%%%%%%%%%%%%%%%%%%%%%%%%%%%%%%%%%%%%%%%%%%%%%%%%%%%%%%%%%%%%%%%%%%%%%%%%%%%%%%%%%%%%%%%%%%%%%%%%
%%%%%%%%%%%%%%%%%%%%%%%%%%%%%%%%%%%%%%%%%%%%%%%%%%%%%%%%%%%%%%%%%%%%%% GDM  AND GDW   %%%%%%%%%%%%%%%%%%%%%%%%%%%%%%%%%%%%%%%%%%%%%%%
%%%%%%%%%%%%%%%%%%%%%%%%%%%%%%%%%%%%%%%%%%%%%%%%%%%%%%%%%%%%%%%%%%%%%%%%%%%%%%%%%%%%%%%%%%%%%%%%%%%%%%%%%%%%%%%%%%%%%%%%%%%%%%%%
Recall that $T\in\mathcal{B}(X)$ is said to have the single
valued extension property at $\lambda_{0}\in\mathbb{C}$ (SVEP for short) if
for every  open neighbourhood   $U\subseteq \mathbb{C}$ of
$\lambda_{0}$, the only  analytic function  $f: U\longrightarrow
X$ which satisfies
the equation $(T-zI)f(z)=0$ for all $z\in U$ is the function $f\equiv 0$. An operator $T$ is said to have the SVEP if $T$ has the SVEP for
every $\lambda\in\mathbb{C}$. Obviously, every operator $T\in\mathcal{B}(X)$ has the SVEP at every $\lambda\in\rho(T)=\mathbb{C}\setminus\sigma(T)$, hence $T$ and $T^*$ have the SVEP at every point of the boundary  $\partial( \sigma(T))$ of the spectrum. Also, we have the implication
$$a(T)<\infty\Longrightarrow T \mbox{ has SVEP at } 0.$$ $$d(T)<\infty\Longrightarrow T^* \mbox{ has SVEP at } 0.$$

In \cite{CZ}, the authors gave some examples showing that $\sigma_{gD\mathcal{M}}(T)\subset\sigma_{gD\mathcal{W}+}(T)$, $\sigma_{gD\mathcal{Q}}(T)\subset\sigma_{gD\mathcal{W}-}(T)$ and
$\sigma_{gD}(T)\subset\sigma_{gD\mathcal{W}}(T)$ can be proper. In the following results we give serval necessary and sufficient conditions for $T$ to have equality.
\begin{proposition}\label{acv}
Let $T\in \mathcal{B}(X)$, then
$\sigma_{gD\mathcal{M}}(T)=\sigma_{gD\mathcal{W}+}(T)$ if and only if $T$ has SVEP at every $\lambda\notin\sigma_{gD\mathcal{W}+}(T)$
\end{proposition}
\begin{proof}
Assume  that $T$ has SVEP at every $\lambda\notin\sigma_{gD\mathcal{W}+}(T)$. If  $\lambda\notin \sigma_{gD\mathcal{W}+}(T)$, then  $T-\lambda I$ is
generalized Drazin upper semi-Weyl, then there exists $(M,N)\in Red(T)$ such that $(T-\lambda I)_{|M}$ is semi-regular  and  $(T-\lambda I)_{|N}$ is quasi-nilpotent. $T$ has SVEP  at every $\lambda\notin\sigma_{gD\mathcal{W}+}(T)$, it follows that  $(T-\lambda I)_{|M}$ has the SVEP at $0$, then $(T-\lambda I)_{|M}$ is bounded below. Hence  $T-\lambda I$ is generalized Drazin bounded below, $\lambda\notin \sigma_{gD\mathcal{M}}(T)$, and since the reverse implication  holds for every operator we conclude that  $\sigma_{gD\mathcal{M}}(T)=\sigma_{gD\mathcal{W}+}(T)$. Conversely, suppose that   $\sigma_{gD\mathcal{M}}(T)=\sigma_{gD\mathcal{W}+}(T)$.  If $\lambda\notin\sigma_{gD\mathcal{W}+}(T)$ then $T-\lambda I$ is generalized Drazin bounded below so $H_0(T-\lambda I)$ is closed. By \cite[Theorem 1.7]{AP},  $T$ has  SVEP at $\lambda$.
\end{proof}

We denote by $\sigma_{lB}(T)$ and $\sigma_{lW}(T)$ respectively the lower Browder and lower Weyl spectra. In the same way we have the following result.

\begin{proposition}\label{acvv}
Let $T\in \mathcal{B}(X)$, then
$\sigma_{gD\mathcal{Q}}(T)=\sigma_{gD\mathcal{W}-}(T)$ if and only if $T^*$ has SVEP at every $\lambda\notin\sigma_{gD\mathcal{W}-}(T)$
\end{proposition}
\begin{proof}
Suppose that $T$ has SVEP at every $\lambda\notin\sigma_{gD\mathcal{W}-}(T)$. If  $\lambda\notin \sigma_{gD\mathcal{W}-}(T)$, then by \cite[Theorem 3.7]{CZ},   $T-\lambda I$ admits GKD and $\lambda\notin acc\sigma_{lW}(T)$. $T^*$ has SVEP  at every $\lambda\notin\sigma_{gD\mathcal{W}-}(T)$,  then  $T^*$ has SVEP  at every $\lambda\notin\sigma_{lW}(T)$, and so $\sigma_{lB}(T)=\sigma_{lW}(T)$. Then  $\lambda\notin acc\sigma_{lB}(T)$. Therefore, $T-\lambda I$ is generalized Drazin surjective according to \cite[Theorem 3.7]{CZ}, $\lambda\notin\sigma_{gD\mathcal{Q}}(T)$ and since the reverse implication  holds for every operator we conclude that  $\sigma_{gD\mathcal{Q}}(T)=\sigma_{gD\mathcal{W}-}(T)$. Conversely, suppose that   $\sigma_{gD\mathcal{Q}}(T)=\sigma_{gD\mathcal{W}-}(T)$.  If
$\lambda\notin\sigma_{gD\mathcal{W}-}(T)$,  then $T-\lambda I$ is generalized Drazin surjective then  $K(T-\lambda I)$ is closed and complemented with a subspace $N$ in $X$ such that $ N\subseteq H_0(T-\lambda I)$  and $(K(T-\lambda I), N)\in Red(T-\lambda I)$, so $K(T-\lambda)+H_0(T-\lambda)=X$. From \cite[Theorem 1.7]{AP},  $T^*$ has the SVEP at $\lambda$.
\end{proof}

As a consequence of the two previous results we have the following proposition.

\begin{proposition}\label{acvvv}
Let $T\in \mathcal{B}(X)$, then
$\sigma_{gD}(T)=\sigma_{gD\mathcal{W}}(T)$ if and only if $T$ and $T^*$ have the  SVEP at every $\lambda\notin\sigma_{gD\mathcal{W}}(T)$
\end{proposition}

A bounded linear operator $T$  is said to satisfy Browder's theorem if $\sigma_W(T)=\sigma_B(T)$, or equivalently $acc\sigma(T)\subseteq \sigma_W(T)$, where $\sigma_B(T)$ is the Browder spectrum of $T$.\\
It is known from \cite{Aie2} that   a-Browder's theorem holds for $T$ if $\sigma_{uW}(T)=\sigma_{uB}(T)$, or equivalently $acc\sigma_{ap}(T)\subseteq \sigma_{uW}(T)$, where $\sigma_{uB}(T)$ and $\sigma_{uW}(T)$   are the upper semi-Browder and upper semi-Weyl spectra of $T$.\\

The following  result shows that Browder's (a-Browder's) theorem holds for $T$ precisely
when  $\sigma_{gD}(T)=\sigma_{gD\mathcal{W}}(T)$ ( $\sigma_{gD\mathcal{M}}(T)=\sigma_{gD\mathcal{W}+}(T)$), which give  new characterizations for Browder's and a-Browder's theorems.
\begin{theorem}\label{rrr}
Let $T\in \mathcal{B}(X)$, then\\
1) a-Browder's theorem holds for $T$ if and only if $\sigma_{gD\mathcal{M}}(T)=\sigma_{gD\mathcal{W}+}(T)$.\\
2) a-Browder's theorem holds for $T^*$ if and only if $\sigma_{gD\mathcal{Q}}(T)=\sigma_{gD\mathcal{W}-}(T)$.\\
3) Browder's theorem holds for $T$ if and only if $\sigma_{gD}(T)=\sigma_{gD\mathcal{W}}(T)$.\\
\end{theorem}

\begin{proof}
1) Suppose that a-Browder's theorem holds for $T$ implies $\sigma_{uB}(T)=\sigma_{uW}(T)$.\\
Using \cite[Theorem 3.4 and Theorem 3.6]{CZ},  we conclude that
\begin{eqnarray*}
\lambda\notin \sigma_{gD\mathcal{M}}(T) & \Longleftrightarrow &  T-\lambda I \mbox{ is generalized Drazin bounded below}  \\
                                         &\Longleftrightarrow &  T-\lambda I \mbox{ admits a GKD and }  \lambda\notin acc\sigma_{uB}(T) \\
                                         & \Longleftrightarrow & T-\lambda I \mbox{ admits a GKD and } \lambda\notin acc\sigma_{uW}(T) \\
                                         & \Longleftrightarrow & T-\lambda I \mbox{ is generalized Drazin upper semi-Weyl }\\
                                         & \Longleftrightarrow & \lambda\notin \sigma_{gD\mathcal{W}+}(T).
                                        \end{eqnarray*}
Hence $\sigma_{gD\mathcal{M}}(T)=\sigma_{gD\mathcal{W}+}(T)$. Conversely, if  $\sigma_{gD\mathcal{M}}(T)=\sigma_{gD\mathcal{W}+}(T)$, from  Proposition \ref{acv},   $T$ has SVEP at every $\lambda\notin\sigma_{gD\mathcal{W}+}(T)$. Since $\sigma_{gD\mathcal{W}+}(T)\subseteq \sigma_{uW}(T)$, $T$ has SVEP at every $\lambda\notin\sigma_{uW}(T)$, so a-Browder's theorem holds for $T$, see \cite[Theorem 4.34]{Aie2}.\\

2) Suppose that a-Browder's theorem holds for $T^*$ then $\sigma_{lB}(T)=\sigma_{lW}(T)$.\\
Using \cite[Theorem 3.4 and  Theorem 3.7]{CZ} we have
\begin{eqnarray*}
\lambda\notin \sigma_{gD\mathcal{Q}}(T) & \Longleftrightarrow &  T-\lambda I \mbox{ is generalized Drazin surjective}  \\
                                         &\Longleftrightarrow &  T-\lambda I \mbox{ admits a GKD and }  \lambda\notin acc\sigma_{lB}(T) \\
                                         & \Longleftrightarrow & T-\lambda I \mbox{ admits a GKD and } \lambda\notin acc\sigma_{lW}(T) \\
                                         & \Longleftrightarrow & T-\lambda I \mbox{ is generalized Drazin lower semi-Weyl }\\
                                         & \Longleftrightarrow & \lambda\notin \sigma_{gD\mathcal{W}-}(T).
                                        \end{eqnarray*}
Hence $\sigma_{gD\mathcal{Q}}(T)=\sigma_{gD\mathcal{W}-}(T)$. Conversely, if  $\sigma_{gD\mathcal{Q}}(T)=\sigma_{gD\mathcal{W}-}(T)$, from  Proposition \ref{acvv},   $T^*$ has SVEP at every $\lambda\notin\sigma_{gD\mathcal{W}-}(T)$. Since $\sigma_{gD\mathcal{W}-}(T)\subseteq \sigma_{lW}(T)$, $T^*$ has SVEP at every $\lambda\notin\sigma_{lW}(T)$, so a-Browder's theorem holds for $T^*$, see \cite[Theorem 4.34]{Aie2}.\\

3) Suppose that Browder's theorem holds for $T$ then $\sigma_{B}(T)=\sigma_{W}(T)$.\\
Using \cite[Theorem 3.4 and  Theorem 3.9]{CZ} we have
\begin{eqnarray*}
\lambda\notin \sigma_{gD}(T) & \Longleftrightarrow &  T-\lambda I \mbox{ is generalized Drazin invertible}  \\
                                         &\Longleftrightarrow &  T-\lambda I \mbox{ admits a GKD and }  \lambda\notin acc\sigma_{B}(T) \\
                                         & \Longleftrightarrow & T-\lambda I \mbox{ admits a GKD and } \lambda\notin acc\sigma_{W}(T) \\
                                         & \Longleftrightarrow & T-\lambda I \mbox{ is generalized Drazin Weyl }\\
                                         & \Longleftrightarrow & \lambda\notin \sigma_{gD\mathcal{W}}(T).
                                        \end{eqnarray*}
Hence $\sigma_{gD}(T)=\sigma_{gD\mathcal{W}}(T)$. Conversely, if  $\sigma_{gD}(T)=\sigma_{gD\mathcal{W}}(T)$, from  Proposition \ref{acvvv},   $T$ and $T^*$ has SVEP at every $\lambda\notin\sigma_{gD\mathcal{W}}(T)$. Since $\sigma_{gD\mathcal{W}}(T)\subseteq \sigma_{W}(T)$, $T$ has SVEP at every $\lambda\notin\sigma_{W}(T)$, so Browder's theorem holds for $T$, see \cite[Theorem 4.23]{Aie2}.\\
\end{proof}
It will be said that generalized Browder's theorem holds for $T\in \mathcal{B}(X)$ if $\sigma_{BW}(T)=\sigma(T)\backslash \Pi(T)$, equivalently,
$\sigma_{BW}(T)=\sigma_D(T)$ \cite{Aiee}, where $\Pi(T)$ is  the set of all poles of the resolvent of $T$.
A classical  result of the second author and H. Zguitti  \cite[Theorem 2.1]{AZ1} shows  that Browder's theorem and generalized Browder's theorem are equivalent.  According to  the previous results and the equivalent between  Browder's theorem and generalized Browder's theorem  \cite[Theorem 2.1]{AZ1} we have the following theorem.
\begin{theorem}
Let $T\in \mathcal{B}(X)$. The statements are equivalent: \\
1) Browder's theorem holds for $T$;\\
2) generalized Browder's theorem holds for $T$;\\
3) $T$ and $T^µ$  have SVEP at every $\lambda\notin\sigma_{gD\mathcal{W}}(T)$;\\
4) $\sigma_{gD}(T)=\sigma_{gD\mathcal{W}}(T)$.
\end{theorem}
In the same way we have the following result.

\begin{theorem}
Let $T\in \mathcal{B}(X)$. The statements are equivalent: \\
1) a-Browder's theorem holds for $T$;\\
2) generalized a-Browder's theorem holds for $T$;\\
3) $T$ has SVEP at every $\lambda\notin\sigma_{gD\mathcal{W}+}(T)$;\\
4) $\sigma_{gD\mathcal{M}}(T)=\sigma_{gD\mathcal{W+}}(T)$.
\end{theorem}

We denote by $\sigma_{lf}(T)$ and $\sigma_{uf}(T)$, $T\in\mathcal{B}(X)$,  respectively the lower and upper semi-Fredholm spectra.
Concerning the pseudo upper/lower B-Fredholm spectrum and the generalized Drazin bounded below/surjective spectrum, we have the following characterization. Note that $\sigma_{puBF}(T)\subset\sigma_{gD\mathcal{M}}(T)$,  $\sigma_{plBF}(T)\subset\sigma_{gD\mathcal{Q}}(T)$ and $\sigma_{pBF}(T)\subset\sigma_{gD}(T)$ are strict \cite{CZ}.
\begin{theorem}\label{eee}
Let $T\in \mathcal{B}(X)$. The statements are equivalent: \\
1) $\sigma_{uf}(T)=\sigma_{uB}(T)$;\\
2) $T$ has SVEP at every $\lambda\notin\sigma_{uf}(T)$;\\
3) $T$  has SVEP at every $\lambda\notin\sigma_{puBF}(T)$;\\
4) $\sigma_{gD\mathcal{M}}(T)=\sigma_{puBF}(T)$.
\end{theorem}
\begin{proof}
$1)\Longleftrightarrow 2)$:
Suppose that $T$  has SVEP at every $\lambda\notin\sigma_{uf}(T)$. If  $\lambda\notin \sigma_{uf}(T)$,  $T-\lambda I$ is
upper semi-Fredholm.  $T$ has  SVEP at $\lambda$, then  $a(T-\lambda I)<\infty$, see \cite[Theorem 3.16]{Aie}. So $\lambda\notin \sigma_{uB}(T)$.
Now,  Suppose that  $\sigma_{uf}(T)=\sigma_{uB}(T)$. Let $\lambda\notin\sigma_{uf}(T)$,   $\lambda\notin\sigma_{uB}(T)$ then
$a(T-\lambda I)<\infty$,  hence  $T$  has SVEP at $\lambda$ by \cite{Aie}.

$3)\Longleftrightarrow 4)$:
Suppose that $T$  has SVEP at every $\lambda\notin\sigma_{puBF}(T)$. If  $\lambda\notin\sigma_{puBF}(T)$, $T-\lambda I$ is
pseudo upper B-Fredholm, then there exists $(M,N)\in Red(T)$ such that $(T-\lambda I)_{|M}$ is semi-regular  and  $(T-\lambda I)_{|N}$ is quasinilpotent. $T$  has SVEP at every $\lambda\notin\sigma_{puBF}(T)$ implies $(T-\lambda I)_{|M}$ has the SVEP at $0$, it follows that  $(T-\lambda I)_{|M}$ is bounded below. Hence  $T-\lambda I$ is generalized Drazin bounded below, $\lambda\notin \sigma_{gD\mathcal{M}}(T)$, and since the reverse implication  holds for every operator we conclude that  $\sigma_{gD\mathcal{M}}(T)=\sigma_{puBF}(T)$. Conversely, assume that   $\sigma_{gD\mathcal{M}}(T)=\sigma_{puBF}(T)$.  If $\lambda\notin\sigma_{puBF}(T)$ then $T-\lambda I$ is generalized Drazin bounded below so $H_0(T-\lambda I)$ is closed. By \cite[Theorem 1.7]{AP},  $T$ has the SVEP at $\lambda$.

$1)\Longleftrightarrow 4)$:  Suppose that  $\sigma_{uf}(T)=\sigma_{uB}(T)$.\\
According to  \cite[Theorem 3.6 and  Theorem 3.4]{CZ} we have
\begin{eqnarray*}
\lambda\notin \sigma_{gD\mathcal{M}}(T) & \Longleftrightarrow &  T-\lambda I \mbox{ is generalized Drazin bounded below}  \\
                                         &\Longleftrightarrow &  T-\lambda I \mbox{ admits a GKD and }  \lambda\notin acc\sigma_{uB}(T) \\
                                         & \Longleftrightarrow & T-\lambda I \mbox{ admits a GKD and } \lambda\notin acc\sigma_{uf}(T) \\
                                         & \Longleftrightarrow & T-\lambda I \mbox{ is pseudo upper B-Fredholm }\\
                                         & \Longleftrightarrow & \lambda\notin \sigma_{puBF}(T).
                                        \end{eqnarray*}
Hence $\sigma_{gD\mathcal{M}}(T)=\sigma_{puBF}(T)$. Conversely, if  $\sigma_{gD\mathcal{M}}(T)=\sigma_{puBF}(T)$, then by  $3)\Longleftrightarrow 4)$, $T$ has SVEP at every $\lambda\notin\sigma_{puBF}(T)$. Since $\sigma_{puBF}(T)\subseteq \sigma_{uf}(T)$, $T$ has SVEP at every $\lambda\notin\sigma_{uf}(T)$,  $1)\Longleftrightarrow 2)$  gives the result.

\end{proof}

\begin{theorem}\label{fff}
Let $T\in \mathcal{B}(X)$. The statements are equivalent: \\
1) $\sigma_{lf}(T)=\sigma_{lB}(T)$;\\
2) $T^*$ has SVEP at every $\lambda\notin\sigma_{lf}(T)$;\\
3) $T^*$  has SVEP at every $\lambda\notin\sigma_{plBF}(T)$;\\
4) $\sigma_{gD\mathcal{Q}}(T)=\sigma_{plBF}(T)$.
\end{theorem}
\begin{proof}
$1)\Longleftrightarrow 2)$:
Suppose that $T^*$  has SVEP at every $\lambda\notin\sigma_{lf}(T)$.  $\lambda\notin \sigma_{lf}(T)$ implies that   $T-\lambda I$ is
lower semi-Fredholm.  $T^*$ has  SVEP at $\lambda$, then  $d(T-\lambda I)<\infty$, see \cite[Theorem 3.17]{Aie}. So $\lambda\notin \sigma_{lB}(T)$.
Now,  Suppose that  $\sigma_{lf}(T)=\sigma_{lB}(T)$. Let $\lambda\notin\sigma_{lf}(T)$,   $\lambda\notin\sigma_{lB}(T)$ then
$d(T-\lambda I)<\infty$,  hence  $T^*$  has SVEP at $\lambda$ by \cite{Aie}.

$3)\Longleftrightarrow 4)$:
Suppose that $T^*$ has SVEP at every $\lambda\notin\sigma_{plBF}(T)$. If  $\lambda\notin \sigma_{plBF}(T)$,  $T-\lambda I$ admits GKD
and $\lambda\notin acc\sigma_{lf}(T)$ by \cite[Theorem 3.4]{CZ}. $T^*$ has SVEP  at every $\lambda\notin\sigma_{plBF}(T)$, it follows that  $T^*$ has SVEP  at every $\lambda\notin\sigma_{lf}(T)$, then $\sigma_{lB}(T)=\sigma_{lf}(T)$ so  $\lambda\notin acc\sigma_{lB}(T)$. Therefore, $T-\lambda I$ is generalized Drazin surjective \cite[Theorem 3.7]{CZ}, $\lambda\notin\sigma_{gD\mathcal{Q}}(T)$ and since the reverse implication  holds for every operator we conclude that  $\sigma_{gD\mathcal{Q}}(T)=\sigma_{plBF}(T)$. Conversely, suppose that   $\sigma_{gD\mathcal{Q}}(T)=\sigma_{plBF}(T)$,  if
$\lambda\notin\sigma_{plBF}(T)$ then $T-\lambda I$ is generalized Drazin surjective then  $K(T-\lambda I)$ is closed and complemented with a subspace $N$ in $X$ such that $ N\subseteq H_0(T-\lambda I)$  and $(K(T-\lambda I), N)\in Red(T-\lambda I)$, so $K(T-\lambda I)+H_0(T-\lambda I)=X$. From \cite[Theorem 1.7]{AP},  $T^*$ has  SVEP at $\lambda$.

$1)\Longleftrightarrow 4)$:  Suppose that  $\sigma_{lf}(T)=\sigma_{lB}(T)$.\\
According to  \cite[Theorem 3.7 and  Theorem 3.4]{CZ} we have
\begin{eqnarray*}
\lambda\notin \sigma_{gD\mathcal{Q}}(T) & \Longleftrightarrow &  T-\lambda I \mbox{ is generalized Drazin surjective}  \\
                                         &\Longleftrightarrow &  T-\lambda I \mbox{ admits a GKD and }  \lambda\notin acc\sigma_{lB}(T) \\
                                         & \Longleftrightarrow & T-\lambda I \mbox{ admits a GKD and } \lambda\notin acc\sigma_{lf}(T) \\
                                         & \Longleftrightarrow & T-\lambda I \mbox{ is pseudo lower B-Fredholm }\\
                                         & \Longleftrightarrow & \lambda\notin \sigma_{plBF}(T).
                                        \end{eqnarray*}
Hence $\sigma_{gD\mathcal{Q}}(T)=\sigma_{plBF}(T)$. Conversely, if  $\sigma_{gD\mathcal{Q}}(T)=\sigma_{plBF}(T)$, by  $3)\Longleftrightarrow 4)$, $T^*$ has SVEP at every $\lambda\notin\sigma_{plBF}(T)$. Since $\sigma_{plBF}(T)\subseteq \sigma_{lf}(T)$, $T$ has SVEP at every $\lambda\notin\sigma_{lf}(T)$, according to $1)\Longleftrightarrow 2)$ we obtain the result.
\end{proof}
As a direct consequence of the Theorem \ref{eee} and  Theorem \ref{fff} we have the following corollary.

\begin{corollary}\label{dfg}
Let $T\in \mathcal{B}(X)$. The statements are equivalent: \\
1) $\sigma_{e}(T)=\sigma_{B}(T)$;\\
2) $T$ and $T^*$ have SVEP at every $\lambda\notin\sigma_{e}(T)$;\\
3) $T$ and $T^*$ have SVEP at every $\lambda\notin\sigma_{BF}(T)$;\\
4) $\sigma_{BF}(T)=\sigma_{D}(T)$;\\
5) $T$ and  $T^*$  have SVEP at every $\lambda\notin\sigma_{pBF}(T)$;\\
6) $\sigma_{gD}(T)=\sigma_{pBF}(T)$.
\end{corollary}

\section{Perturbations}
%%%%%%%%%%%%%%%%%%%%%%%%%%%%%%%%%%%%%%%%%%%%%%%%%%%%%%%%%%%%%%%%%%%%%%%%%
%%%%%%%%%%%%%%%%%%%%%%%%%%%%%%%%%%%%%%%%%%%%%%%%%%%% Small commuting perturbation %%%%%%%%%%%%%%%%%%%%%%%%%%%%
%%%%%%%%%%%%%%%%%%%%%%%%%%%%%%%%%%%%%%%%%%%%%%%%%%%%%%%%%%%%%%%%%%%%%%%%%%%%%%%%%%%%%%%%%%%%%%%%%%%%%%%%%%%%%%%%%%
Now, we consider the classes of operators introduced in \cite{CZ}:
$$gDR :=\{ T\in\mathcal{B}(X); \mbox{ there exists } (M,N)\in Red(T) \mbox{ such that } T_{|M}\in R \mbox{ and } T_{|N} \mbox{ is quasinilpotent }\}.$$
$$DR  :=\{ T\in\mathcal{B}(X); \mbox{ there exists } (M,N)\in Red(T) \mbox{ such that } T_{|M}\in R \mbox{ and } T_{|N} \mbox{ is nilpotent }\}.$$
Where $R$ denote any of the following classes:  bounded below/surjective operators,   upper(lower) semi-Fredholm operators, Fredholm operator, upper(lower) semi-Weyl operators.
\begin{proposition}
Let $T\in \mathcal{B}(X)$. If $T\in gDR$, then there exists $\alpha> 0$  such that for every $S\in\mathcal{B}(X)$ invertible operator,
satisfying $ST=TS$ and  $||S||<\alpha$,  we have   $T-S\in DR$.
\end{proposition}
\begin{proof}
If $T\in gDR$, then  $T$ admits a GKD and $0\in acc\sigma_R(T)$,  see\cite{CZ}. From  \cite[Théorème 2.1]{WB} $T-S$ is semi-regular, and since $acc\sigma_R(T-S)=acc\sigma_R(T)$, $\sigma_R(T)$ the spectrum associated to the classe $R$, then
$T$ is of Kato type and  $0\in acc\sigma_R(T-S)$. According to \cite[Theorem 4.1]{CZ},  $T-S\in DR$
\end{proof}

%%%%%%%%%%%%%%%%%%%%%%%%%%%%%%%%%%%%%%%%%%%%%%%%%%%%%%%%%%%%%%%%%%%%%%%%%%%%%%%%%%%%%%%%%%%%%%%%%%%%%%%%%%%%%%%%%%
%%%%%%%%%%%%%%%%%%%%%%%%%%%%%%%%%%%%%%%%%%%%%%%%%%%%%%%%%%%%%%%%%%%%%%%%%%%%%%%%%%%%%%%%%%%%%%%%%%%%%%%%%%%%%%%%%%%%
%%%%%%%%%%%%%%%%%%%%%%%%%%%%%%%%%%%%%%%%%%%%%%%%%%%%%%%%%%%%%%%%%%%%%%%%%%%%%%%%%%%%%%%%%%%%%%%%%%%%%%%%%%%%%%%%%%%%%%%%%%%
Let $\mathcal{F}(X)$ denote the ideal of finite rank operators on $X.$
A bounded linear operator $F\in\mathcal{B}(X)$ is power finite rank if $F^n\in\mathcal{F}(X)$ for some $n \in \mathbb{N}.$
In what follow, we will prove that  pseudo B-Weyl operators satisfying Browder's theorem  is stable by power finite rank perturbations.

\begin{proposition}
Let $T\in \mathcal{B}(X)$,  $F^n\in\mathcal{F}(X)$ for some $n \in \mathbb{N}$ commutes with $T$ then:

\begin{enumerate}
  \item  If $T$ satisfy Browder theorem, then $\sigma_{gD\mathcal{W}}(T+F)=\sigma_{gD\mathcal{W}}(T)$;
  \item  If $T$ and $T^*$ have SVEP at every $\lambda\notin\sigma_{e}(T)$, then $\sigma_{pBF}(T+F)=\sigma_{pBF}(T)$
\end{enumerate}
\end{proposition}
\begin{proof}
$(1)$ According to \cite[ Theorem 2.2]{ZZY}, we have $acc(\sigma(T))=acc(\sigma(T+F)).$
Then $\lambda\in  \sigma_{gD}(T)$ if and only if $\lambda\in acc(\sigma(T))$
if and only if $\lambda\in acc(\sigma(T+F))$
if and only if $\lambda\in\sigma_{gD}(T+F).$ So $\sigma_{gD}(T+F)=\sigma_{gD}(T).$
The Theorem \ref{rrr} and Theorem \ref{dfg}  give the result.
\end{proof}
By the same argument we have the following proposition.
\begin{proposition}
Let $T\in \mathcal{B}(X)$ satisfy Browder theorem, $Q$ a quasi-nilpotent operator commutes with $T$ then:
 \begin{enumerate}
  \item  If $T$ satisfy Browder theorem, then $\sigma_{gD\mathcal{W}}(T+Q)=\sigma_{gD\mathcal{W}}(T)$;
  \item  If $T$ and $T^*$ have SVEP at every $\lambda\notin\sigma_{e}(T)$, then $\sigma_{pBF}(T+Q)=\sigma_{pBF}(T)$
\end{enumerate}
\end{proposition}
\begin{proof}
Since  $\sigma_{gD}(T+Q)=\sigma_{gD}(T)$,  from Theorem \ref{rrr} and Theorem \ref{dfg} we have the result.
\end{proof}

\begin{remark}
 Let $T\in\mathcal{B}(X)$, we have $\sigma_{pBF}(T)\subset \sigma_e(T),$ $\sigma_{gD\mathcal{W}}(T)\subset \sigma_W(T)$ and $\sigma_{gD}(T)\subset \sigma_D(T)$ but generally these inclusions are proper. Indeed, let $T$ and $S$ defined on $l^2(\mathbb{N})$ by
 $$T(x_1, x_2, x_3,...)=(\frac{1}{2}x_2, \frac{1}{3}x_3,...); \>\>\>\>\> S(x_1, x_2, x_3,...)=(0, \frac{1}{2}x_1, 0, 0, ...).$$
 Then $T$ is quasi-nilpotent with infinite ascent and hence $$\sigma_{gD}(T)=\emptyset \mbox{ but }\sigma_D(T)=\{0 \}.$$
  Furthermore,
  $$\sigma_{pBF}(S)=\sigma_{gD\mathcal{W}}(S)=\emptyset \mbox{ but } \sigma_{e}(S)=\sigma_{W}(S)=\{0 \}.$$
  \end{remark}

The following lemma, will be needed in the sequel to study Riesz perturbations.

\begin{lemma}\label{l2}
Let $T\in \mathcal{B}(X)$.
\begin{enumerate}
  \item $\sigma_D(T)=\sigma_{gD}(T)\cup iso(\sigma_D(T))$;
  \item $\sigma(T)=\sigma_{gD}(T)\cup iso(\sigma(T))$;
  \item $\sigma_{se}(T)=\sigma_{gK}(T)\cup iso(\sigma_{se}(T))$;
  \item $\sigma_{es}(T)=\sigma_{gK}(T)\cup iso(\sigma_{es}(T))$;
  \item $\sigma_{e}(T)=\sigma_{pBF}(T)\cup iso(\sigma_e(T))$;
  \item $\sigma_{W}(T)=\sigma_{gD\mathcal{W}}(T)\cup iso(\sigma_W(T))$.
\end{enumerate}
\end{lemma}
\begin{proof}
 $(1)$ Let $\lambda\in \sigma_D(T)\setminus \sigma_{gD}(T)$, then $T-\lambda$ is a generalized Drazin operator hence there exists an $\varepsilon >0$ such that $T-\mu$ is  Drazin invertible for all $\mu\in D(\lambda, \epsilon)\setminus\{\lambda\}.$ Indeed,
If $T-\lambda$ is a generalized Drazin operator, then there exists  two closed $T$-invariant subspaces $X_{1}$ and $X_{2}$ of $X$
 such that $ X=X_{1}\oplus X_{2}$ and $T-\lambda=(T-\lambda)_{\shortmid X_{1}} \oplus (T-\lambda)_{\shortmid X_{2}}$
 where $(T-\lambda)_{\shortmid X_{1}}$ is invertible and $(T-\lambda)_{\shortmid X_{2}}$ is quasi-nilpotent.
 If $X_{1}=\{0\}$, $T-\lambda$ is quasi-nilpotent,  then  for all $\mu\neq\lambda$,  $T- \mu $ is invertible, hence  $T- \mu$  is  Drazin invertible.
 If $X_{1}\neq\{0\}$ then $(T-\lambda)_{\shortmid X_{1}}$ is invertible,  hence there exists $\varepsilon > 0$ such that $(T-\mu)_{\shortmid X_{1}} $ is invertible for all $\mu\in D(\lambda,\epsilon)$, hence  $T- \mu $ is  Drazin invertible for all $\mu\in D(\lambda,\epsilon)$.
 As $(T-\lambda)_{\shortmid X_{2}}$ is quasi-nilpotent, then for all $\mu\neq \lambda$ $(T-\mu)_{\shortmid X_{2}}$ is invertible and hence $(T-\mu)_{\shortmid X_{2}}$ is  Drazin invertible  for all $\mu\in D(\lambda,\varepsilon)\setminus\{\lambda\}.$ Since $(T-\mu)_{\shortmid X_{2}}$ and $(T-\mu)_{\shortmid X_{1}}$ are  Drazin invertible for all $\mu\in D(\lambda,\varepsilon)\setminus\{\lambda\}$, then we get $T-\mu$ is  Drazin invertible  for all $\mu \in D(\lambda, \varepsilon)\setminus\{\lambda\}$. This implies that $$D(\lambda,\epsilon)\setminus\{\lambda\}\cap \sigma_D(T)=\emptyset,$$ hence $\lambda\in iso (\sigma_D(T)).$ Therefore, $$\sigma_D(T)\subseteq\sigma_{gD}(T)\cup iso(\sigma_D(T)).$$ The reverse inclusion is always true.\\
The assertion $(2)$ is clear, since $\sigma_{gD}(T)=acc\sigma(T)$.\\
For $(3),$ let $\lambda\in \sigma_{se}(T)\setminus \sigma_{gK}(T)$,
$T-\lambda$ is a pseudo Fredholm operator.
By \cite[Theorem 2.2]{JZ}, there exists an $\epsilon >0$ such that $T-\mu$ is semi-regular for all $\mu\in D(\lambda, \epsilon)\setminus\{\lambda\}$, this implies that $D(\lambda,\epsilon)\setminus\{\lambda\}\cap \sigma_{se}(T)=\emptyset$, hence $\lambda\in  iso (\sigma_{se}(T))$. Therefore, $\sigma_{se}(T)\subseteq\sigma_{gK}(T)\cup iso(\sigma_{se}(T))$, the opposite  inclusion is always  true.\\
To prove $(4),$ let $\lambda\in \sigma_{es}(T)\setminus \sigma_{gK}(T)$,
$T-\lambda$ is a pseudo Fredholm operator.
By \cite[Theorem 2.2]{JZ}, there exists an $\epsilon >0$ such that $T-\mu$ is semi-regular for all $\mu\in D(\lambda, \epsilon)\setminus\{\lambda\}$,  hence $T-\mu$ is essentially semi-regular for all $\mu\in D(\lambda, \epsilon)\setminus\{\lambda\}$,  this implies that $D(\lambda,\epsilon)\setminus\{\lambda\}\cap \sigma_{es}(T)=\emptyset$, thus $\lambda\in  iso (\sigma_{se}(T))$. Therefore, $\sigma_{es}(T)\subseteq\sigma_{gK}(T)\cup iso(\sigma_{es}(T))$, since $\sigma_{gK}(T)\subseteq\sigma_{es}(T)$, we have $$\sigma_{es}(T)=\sigma_{gK}(T)\cup iso(\sigma_{es}(T)).$$
For the assertion $(5),$ let $\lambda\in \sigma_e(T)\setminus \sigma_{pBF}(T)$, then $T-\lambda$ is a pseudo B-Freholm operator, hence there exists an $\epsilon >0$ such that $T-\mu$ is Fredholm for all $\mu\in D(\lambda, \epsilon)\setminus\{\lambda\}$.
Indeed, without loss of generality we can assume that $\lambda=0.$
If $T$ is pseudo B-Fredholm, then there  exists two closed $T$-invariant  subspaces $X_1$ and $X_2$  such that $ X=X_{1}\oplus X_{2}$; $T_{\shortmid X_{1}}$ is Fredholm, $T_{\shortmid X_{2}}$ is quasi-nilpotent  and $T=T_{\shortmid X_{1}} \oplus T_{\shortmid X_{2}}$.\\
If $X_{1}=\{0\}$, $T$ is quasi-nilpotent, hence $\mu I - T$ is invertible for all $\mu \neq 0$, that is $\mu I - T$ is Fredholm  for all $\mu \neq 0$.\\
If $X_{1}\neq\{0\},$ then $T_{\shortmid X_{1}}$ is Fredholm, hence there exists
$\varepsilon > 0$ such that $(\mu I - T)_{\shortmid X_{1}} $ is Fredholm for all $\mu\in D(0,\epsilon)$.
As $T_{\shortmid X_{2}}$ is quasi-nilpotent, then for all $\mu\neq 0$,  $(\mu I - T)_{\shortmid X_{2}}$ is invertible, then $(\mu I -T)_{\shortmid X_{2}}$ is Fredholm  for all $\mu\in D^*(0,\varepsilon)$. Since $(\mu I-T)_{\shortmid X_{2}}$ and $(\mu I-T)_{\shortmid X_{1}}$ are Fredholm for all $\mu\in D^*(0,\varepsilon)$, we have $\mu I -T$ is  Fredholm for all $\mu \in D^*(0, \varepsilon)$.\\ This implies that $D(\lambda,\epsilon)\setminus\{\lambda\}\cap \sigma_{e}(T)=\emptyset$, hence $\lambda\in  iso (\sigma_{e}(T))$. Therefore, $$\sigma_{e}(T)\subseteq\sigma_{pBF}(T)\cup iso(\sigma_e(T)).$$ Since the opposite  inclusion is true, then we conclude $(5).$\\
By a similar argument as in $(5)$, we can prove $(6).$
\end{proof}

\begin{theorem}\label{rr}
Let $T\in \mathcal{B}(X)$ and  $R\in \mathcal{B}(X)$ be a Riesz operator which commutes with $T$. Then the following statements hold:
\begin{enumerate}
  \item If $iso(\sigma_e(T))=\emptyset$, then $\sigma_{pBF}(T+R)=\sigma_{pBF}(T)$;
  \item If $iso(\sigma_W(T))=\emptyset$, then $\sigma_{gD\mathcal{W}}(T+R)=\sigma_{gD\mathcal{W}}(T)$.
\end{enumerate}
\end{theorem}
\begin{proof}
To prove $(1),$  we have $\sigma_e(T+R)=\sigma_e(T)$ and since $iso(\sigma_e(T))=\emptyset $, then by lemma \ref{l2}, we get
 $\sigma_{pBF}(T)=\sigma_e(T),$ hence $\sigma_{pBF}(T+R)=\sigma_{pBF}(T).$\\
For the assertion $(2),$  we have $\sigma_W(T+R)=\sigma_W(T)$ and since $iso(\sigma_W(T))=\emptyset,$ then by lemma \ref{l2}, we have
 $\sigma_{gD\mathcal{W}}(T)=\sigma_W(T)$, hence $\sigma_{gD\mathcal{W}}(T+R)=\sigma_{pBW}(T)$.
 \end{proof}
%%%%%%%%%%%%%%%%%%%%%%%%%%%%%%%%%%%%%%%%%%%%%%%%%%%%%%%%%%%%%%%%%%%%%%%%%%%%%%%%%%%%%%%%%%%%%%%%%%%%%%%%%%%%%%%%%%%%%%%
%%%%%%%%%%%%%%%%%%%%%%%%%%%%%%%%%%%%%%%%%%%%%%%%%%%%%%%%%%%%%%%%%%%%%%%%%%%%%%%%%%%%%%%%%%%%%%%%%%%%%%%%%%%%%%%%%%%%%%%%%%%%%
Note that the essential quasi-Fredholm spectrum is not stable under commuting quasi-nilpotent and compact perturbations,
hence it is not stable under commuting Riesz perturbation, see \cite{Mbe2}.
\begin{theorem}
Let $T\in\mathcal{B}(X).$
\begin{enumerate}
\item   If $iso(\sigma_{es}(T))=\emptyset$ and $R$ is a Riesz operator such that $TR=RT,$ then
  $$\sigma_{gK}(T+R)=\sigma_{gK}(T)\mbox{  and  } \sigma_{eq}(T+R)=\sigma_{eq}(T).$$
  \item  If $iso(\sigma_{se}(T))=\emptyset$   and $Q$ is a quasi-nilpotent operator such that $QT=TQ,$ then
  $$\sigma_{gK}(T+Q)=\sigma_{gK}(T)\mbox{  and  } \sigma_{eq}(T+Q)=\sigma_{eq}(T)$$
\end{enumerate}
\end{theorem}
\begin{proof}
  To prove $(1),$ since $\sigma_{gK}(T)\subseteq \sigma_{eq}(T)\subseteq \sigma_{es}(T)$, then by part  (4) of Lemma \ref{l2},
    we have  $$\sigma_{eq}(T)\cup iso(\sigma_{es}(T))=\sigma_{es}(T).$$
     According to \cite[Corollary 17]{KMR}, if $R$  is a Riesz operator commutes with $T$, then we have $\sigma_{es}(T+R)=\sigma_{es}(T).$ By hypothesis
     $ iso(\sigma_{es}(T))=\emptyset,$ hence $\sigma_{es}(T)=\sigma_{eq}(T)$  and by Lemma \ref{l2}, we have $\sigma_{es}(T)=\sigma_{gK}(T)$. This gives the result.\\
To prove $(2),$ from \cite{Mbe2}, if $Q$ is a quasi-nilpotent operator, then we have $\sigma_{se}(T+Q)=\sigma_{se}(T).$
 By hypothesis $ iso(\sigma_{se}(T))=\emptyset,$ hence by Lemma \ref{l2}, we have that $\sigma_{se}(T)=\sigma_{gK}(T)$. This gives the result.
\end{proof}

%%%%%%%%%%%%%%%%%%%%%%%%%%%%%%%%%%%%%%%%%%%%%%%%%%%%%%%%%%%%%%%%%%%%%%%%%%%%%%%%%%%%%%%%%%%%%%%%%%%
%%%%%%%%%%%%%%%%%%%%%%%%%%%%%%%%%%%%%%%%%%%%%%%%%%%%%%%%%%%%%%%%%%%%%%%%%%%%%%%%%%%%%%%%%%%%%%%%%%%%%%%%%%%%%%%%%
\begin{example}
Let  $T$ is an unilateral weighted right shift on $l^p(\mathbb{N})$, $1\leq p< \infty$, with weight sequence $(w_{n})_{n\in\mathbb{N}}$. If $lim_{n\rightarrow\infty }inf(w_1 ....w_n)^{1/n}=0,$ then $T$ and $T^*$ have the SVEP and by \cite[corollary 3.118]{Aie}:
$$\sigma_{su}(T)=\sigma_{ap}(T)=\sigma_{se}(T)=\sigma_{e}(T)=\sigma_{W}(T)=\sigma(T)=\textbf{D}(0, r(T)),$$  where  $\textbf{D}(0, r(T))$ the closed disc,
 hence $iso(\sigma(T))=iso(\sigma_W(T))=iso(\sigma_e(T))=\emptyset$. If $R$ a Riesz operator which commutes with $T$ and $Q$ a quasi-nilpotent operator commutes with $T$, then:
\begin{center}
$\sigma_{gK}(T+Q)=\sigma_{gK}(T)$;
 $\sigma_{pBF}(T+R)=\sigma_{pBF}(T)$;
  $\sigma_{pBW}(T+R)=\sigma_{pBW}(T)$.$\Box$
\end{center}
\end{example}

%%%%%%%%%%%%%%%%%%%%%%%%%%%%%%%%%%%%%%%%%%%%%%%%%%
%%%%%%%%%%%%%%%%%%%%%%%%%%%%%%%%%%%%%%%%%%%%%%%%%%%%%%%%%%%%%%%%%%%%%%%%%%%%%%%%%%%%%%%%%%%%%%%%%%
%%%%%%%%%% REMARK %%%%%%%%%%%%%%%%%%%%%%%%%%%%%%%%%%%%%%%%%%%%%%%%%%%%%%%%%%%%%

\section{Commutator and Pseudo B-Fredholm Perturbations }

Let $T,S\in\mathcal{B}(X)$, denote by $[T,S]$ the commutator of $T$ and $S.$\\
In what follows, we prove that we can perturb a pseudo B-Fredholm operator $T\in\mathcal{B}(X)$
by a bounded operator $S$ satisfying $[T,S]=0$ to obtain a
Fredholm operator  $T+S.$

\begin{proposition}
Let $T\in\mathcal{  B}(X)$ a pseudo B-Fredholm operator. Then  there exists $S\in \mathcal{B}(X)$ such that :
\begin{center}
$T+S$ is Fredholm, $TS$ is quasi-nilpotent and $[T,S]=0.$
\end{center}
\end{proposition}

\begin{proof}
If $T$ is pseudo B-Fredholm, then there exists two closed $T$-invariant subspaces $X_{1}$ and $X_{2}$ of $X$
such that $X=X_1\oplus X_2$ and $T_1=T_{\shortmid X_{1}}$ is upper semi-Fredholm and $T_2=T_{\shortmid X_{2}}$ is quasi-nilpotent.
Let $S=0\oplus (I_2-T_2)$,   $I_2=I_{|X_{2}}$.  Since  $T_1$ is Fredholm, then $T+S=T_1\oplus I_2$ is a Fredholm operator. We have : \\
\begin{eqnarray*}
  TS &=& [T_1\oplus T_2] [0\oplus (I_2-T_2)] \\
   &=& T_2(I_2-T_2)= (I_2-T_2)T_2\\
   &=& [0\oplus (I_2-T_2)][T_1\oplus T_2]=ST
\end{eqnarray*}
 From the well known spectral radius formula $$r(TS)=r((I_2-T_2)T_2)\leq r(I_2-T_2)r(T_2)=0$$
 Therefore $TS$ is quasinilpotent.
\end{proof}
%%%%%%%%%%%%%%%%%%%%%%%%%%%%%%%%%%%%%%%%%%%%%%%%%%%%%%%%%%%%%%%%%%%%%%%%%%%%%%%%%%%%%%%%%%%%%%%%%%%%%
%%%%%%%%%%%%%%%%%%%%%%%%%%%%%%%%%%%%%%%%%%%%%%%%%%%%%%%%%%%%%%%%%%%%%%%%%%%%%%%%%%%%%%%%%%%%%%%%%%
%%%%%%%%%%%%%%%%%%%%%%%%%%%%%%%%%%%%%%%%%%%%%%%%%%%%%%%%%%%%%%%%%%%%%%%%%%%%%%%%%%%%%%%%%%%%%%%%%%%%%%%%%
In what follows, we prove that we can perturb a pseudo Fredholm operator $T\in\mathcal{B}(X)$
by a bounded operator $S$ satisfying $[T,S]=0$ to obtain a
semi-regular operator  $T+S.$
\begin{proposition}
Let $T\in\mathcal{  B}(X)$ a pseudo Fredholm operator. Then  there exists $S\in \mathcal{B}(X)$ such that :
\begin{center}
$T+S$ is semi-regular, $TS$ is quasi-nilpotent and $[T,S]=0$.
\end{center}
\end{proposition}
\begin{proof}
 $T$ is  a pseudo Fredholm operator, then there exists a subsets $M$ and $N$  of $X$ such that
  \begin{center}
  $X=M\oplus N$ and $T=T_1\oplus T_2$
\end{center}
with  $T_1=T_{|M}$ is a semi-regular operator and $T_2=T_{|N}$ is a quasinilpotent.\\
Let $S=0\oplus (I_2-T_2)$,   $I_2=I_{|N}$.  Since  $T_1$ is semi-regular then    $T+S=T_1\oplus I_2$ is a semi-regular operator. We have : \\
\begin{eqnarray*}
  TS &=& [T_1\oplus T_2] [0\oplus (I_2-T_2)] \\
   &=& T_2(I_2-T_2)= (I_2-T_2)T_2\\
   &=& [0\oplus (I_2-T_2)][T_1\oplus T_2]=ST
\end{eqnarray*}
\end{proof}

In the following, we give a generalization of \cite[Theorem 2.1]{O} and \cite[Proposition 1.1]{LW}.
\begin{theorem}
Let $T\in \mathcal{B}(X)$ a  pseudo B-Weyl operator. Then there exists an operator $F\in \mathcal{B}(X)$ such that $T+\lambda F$ is invertible for all $\lambda\in\mathbb{C}\setminus\{0\}$ and $$[T,F]q(T,F)[T,F]=0$$ Where $q(T,F)$ is any polynomial in $T$ and $F$.
\end{theorem}

\begin{proof}
$T$ is a pseudo B-Weyl operator, then there exists two closed $T$-invariant subspaces $X_{1}$ and $X_{2}$ of $X$
such that $X=X_1\oplus X_2$ and $T_1=T_{\shortmid X_{1}}$ is pseudo B-Weyl and $T_2=T_{\shortmid X_{2}}$ is quasi-nilpotent.
Since $T_1$ is  a Weyl operator $ind(T_1)=0$, according to \cite[Theorem 1.2]{O}, there exists $F_1\in\mathcal{B}(M)$ such that
$T_1+ \lambda F_1$ is invertible and $[T_1,F_1]q(T_1,F_1)[T_1,F_1]=0.$ Where $q(T_1,F_1)$ is any polynomial in $T_1$ and $F_1$.
Set $F=F_1\oplus I_2$ where $I_2$ is the restriction of $I$ to $ X_{2}$. $T_2$ is quasi-nilpotent this implies that $T_2 +\lambda I_2$ is invertible for all $\lambda\neq 0$, hence $T+\lambda F= T_2\oplus T_2+ \lambda(F_1\oplus I_2)=(T_1+ \lambda F_1)\oplus (T_2 +\lambda  I_2)$ is invertible.
On the other hand, $[T_2, I_2]=0$ and $[T_1,F_1]q(T_1,F_1)[T_1,F_1]=0$ therefore $[T,F]q(T,F)[T,F]=0$ where $q(T,F)$ is any polynomial in $T$ and $F$.
\end{proof}

\end{document}